\documentclass[12pt]{amsproc}

\numberwithin{equation}{section}
\usepackage{graphicx}
\usepackage{latexsym}
\usepackage{amsmath}
\usepackage{amssymb}
\usepackage{amsfonts}
\usepackage{verbatim}
\usepackage{mathrsfs}
\usepackage[latin1]{inputenc}
\usepackage{color}

\newtheorem{tm}{Theorem}[section]
\newtheorem{rk}{Remark}[section]

\newtheorem{prop}{Proposition}[section]
\newtheorem{lm}{Lemma}[section]

\newcommand{\ee}{\mathbb E}
\newcommand{\pp}{\mathbb P}
\newcommand{\nn}{\mathbb N}
\newcommand{\hh}{\mathbb H}
\newcommand{\rr}{\mathbb R}

\newcommand{\OO}{\mathcal O}
\newcommand{\EE}{\mathcal E}

\newcommand{\FFF}{\mathscr F}

\newcommand{\<}{\langle}
\renewcommand{\>}{\rangle}

\begin{document}

\title[Approximating SEEs with White and Rough Noises]
{Approximating Stochastic Evolution Equations with Additive White and Rough Noises}

\author{Yanzhao Cao}
\address{Department of Mathematics and Statistics, Auburn University, Auburn, AL 36849}
\email{yzc0009@auburn.edu}
\thanks{}
\author{Jialin Hong}
\address{Academy of Mathematics and Systems Science, Chinese Academy of Sciences, Beijing, China}
\curraddr{}
\email{hjl@lsec.cc.ac.cn}
\thanks{}
\author{Zhihui Liu}
\address{Academy of Mathematics and Systems Science, Chinese Academy of Sciences, Beijing, China}
\curraddr{}
\email{liuzhihui@lsec.cc.ac.cn (Corresponding author)}
\thanks{}

\subjclass[2010]{60H15, 60H35, 65C30, 65M60}

\keywords{parabolic/hyperbolic stochastic partial differential equation, fractional Brownian motion, Hurst index $H\leq \frac{1}{2}$, Wong-Zakai approximation, Galerkin approximation}

\date{\today}

\dedicatory{}
\maketitle

\begin{abstract}
In this paper, we analyze Galerkin approximations  for  stochastic evolution equations driven by an additive Gaussian noise which is temporally white and spatially fractional with Hurst index less than or equal to $1/2$. First we  regularize the noise by the Wong-Zakai approximation and obtain its optimal order of convergence. 
Then we apply the Galerkin method  to discretize the stochastic evolution equations with  regularized noises. 
Optimal error estimates are obtained for the Galerkin approximations. In particular,  our error estimates remove an infinitesimal factor which appears in the error estimates of various numerical methods for stochastic evolution equations in existing literatures. 
\end{abstract}

\section{Introduction}
\label{sec1}

In this paper we consider the Galerkin approximation of the stochastic evolution equation (SEE)
\begin{align}\label{see}
L u(t,x)=b(u(t,x))+\xi(t,x),\quad (t,x)\in I\times \OO, 
\end{align}
with either homogenous Dirichlet boundary condition
\begin{align}\label{dbc}
u(t,0)=u(t,1)=0,\quad t\in I 
\end{align}
or Neumann boundary condition
\begin{align}\label{nbc}
\partial_x u(t,0)=\partial_x u(t,1)=0,\quad t\in I, 
\end{align}
where  $I=[0,T]$ and $\OO=(0,1)$. Here $L$ is a second order partial differential operator, the shift coefficient $b$ is a real-valued Lipschitz continuous function, and $\xi$ is a white-fractional noise, i.e., $\xi=\frac{\partial^2W}{\partial t \partial x}$  where $W=\{W(t,x),\ (t,x)\in I\times \OO\}$ is a fractional Gaussian sheet on a stochastic basis  $(\Omega,\FFF,(\FFF_t)_{t\in I},\pp)$ such that
\begin{align}\label{wfn}
\ee\bigg[W(s,x)W(t,y)\bigg]
=(s\wedge t)\frac{x^{2H}+y^{2H}-|x-y|^{2H}}{2}
\end{align}
for all $(s,x),(t,y)\in I\times \OO$.
Here the parameter $H$ is called the Hurst index. When $H=1/2$, 
the white-fractional noise $\xi$ becomes the standard space-time white noise.  

There have been many studies on numerical approximations of SEEs with space-time white noise or smoother noises (cf. \cite{CP12, FLP14, Hau03, JR15, KLM11, Kru14, WGT14, ZTRK14} and references therein).  In this paper we focus on the case when $H<1/2$, which makes the noise ``rougher" than the white noise  in the spatial dimension. 
In some practical applications such as flows in porous media, such rough noises are more suitable to  model physical properties (cf. \cite{CCCSV13} and references therein).

Though the methodology developed in this paper is applicable to a variety of partial differential operators,  in this study, we focus on  the parabolic partial differential operator $L=L^I=\partial_t-\partial_{xx}$,  which makes Eq. \eqref{see} a stochastic heat equation (SHE),  and the hyperbolic partial differential operator $L=L^{II}=\partial_{tt}-\partial_{xx}$,  which makes Eq. \eqref{see} a stochastic wave equation (SWE). For the SHE we impose the initial condition $u(0,x)=u_0(x)$;  for the  SWE we the impose initial conditions $u(0,x)=u_0(x)$ and $\partial_t u(0,x)=v_0(x)$. 

A key step of our Galerkin approximation for Eq. \eqref{see}  is to apply the approximation
\begin{align}\label{eq:wz}
\tilde \xi(t,x)=\sum_{i=0}^{m-1}\sum_{j=0}^{n-1}
\bigg[\frac{1}{kh}\int_{I_i}\int_{\OO_j}\xi(ds,dy)\bigg]
\chi_{i,j}(t,x),\  (t,x)\in I\times \OO, 
\end{align}
to regularize the white-factional noise $\xi(t,x)$.  
Here $\{I_i\}_{i=0}^{m-1}$ and $\{\OO_j\}_{j=0}^{n-1}$ are uniform  partitions of $I$ and $\OO$, respectively, with grid sizes $k=T/m$ and $h=1/n$,
and $\chi$ is the usual indicator function on these partitions.  
The spatial partition will also serve as the finite element mesh in the Galerkin approximation of  the SHE. We note that $\tilde \xi$ is a piecewise constant process in each time-space grid $I_i\times \OO_j$. In this sense \eqref{eq:wz} is a  Wong-Zakai type approximation for the stochastic process  $\xi(t,x)$ (see \cite{WZ65}  for the origin of the Wong-Zakai approximation). For simplicity, we call \eqref{eq:wz} the Wong-Zakai approximation of $\xi$.   The Wong-Zakai approximation is a commonly used method in the numerical study of stochastic differential equations (cf. \cite{MT04} and references therein).  It has also been widely used in theoretical analysis as well as  numerical approximations of SPDEs.  For instance, based on the Wong-Zakai approximation for the space-time white noise, the authors in \cite{BMS95} obtained a support theorem for the law of the solution of SHEs. For numerical solutions  using the Wong-Zakai approximation, we refer to  \cite{CYY07, CHL15b, CHL15a} for the finite element method  for stochastic elliptic equations and  \cite{ANZ98, CY07, DZ02}  for SEEs.


%

To obtain error estimates for the Galerkin approximation of Eq. \eqref{see}, we first  study the well-posedness and Sobolev regularity of its mild solution (defined in Section \ref{sec2}). Specifically, we  prove that if the initial datum  possesses finite $p$-th moment for $p\geq 2$, then Eq. \eqref{see} has a unique mild solution with  uniformly bounded $p$-th moment.  Moreover,  this mild solution is in $\dot{\hh}^\beta$ (whose definition is given in Section \ref{sec2}) provided $u_0$ is in $\dot{\hh}^\beta$ and/or $v_0$ is in $\dot{\hh}^{\beta-1}$ for any $\beta\in [0,H)$ (see Theorem \ref{wel}). 

The main results of this study are the error estimates for both the approximation to the exact solution of  the SEE through the Wong-Zakai approximation and the numerical  solutions  through Galerkin finite element approximation for the SHE and the spectral Galerkin  approximation for the SWE.  
Let   $u$ be the mild solution of Eq. \eqref{see} and $ \tilde u$ be the mild solution of the SEE with the noise term replaced by the   Wong-Zakai approximation. 
Then  (see Theorem \ref{uu}) for SHE, 
\begin{align}\label{eq:error-W-Z-SHE}
\sup_{t\in I}\bigg(\ee\bigg[ \|u (t)-\tilde u(t)\|_{\mathbb L^2}^p\bigg] \bigg)^\frac1p
\le C \left(h^H+k^\frac14 h^{H-\frac12}\right), 
\end{align}
and for SWE, 
\begin{align}\label{eq:error-W-Z-SWE}
\sup_{t\in I}\bigg(\ee\bigg[ \|u (t)-\tilde u(t)\|_{\mathbb L^2}^p\bigg] \bigg)^\frac1p
\le C \left(h^H+k^\frac12 h^{H-\frac12}\right).
\end{align}
Here and in the rest of the paper, $C$ denotes a generic constant whose value may be different at different appearances. 
A similar result for the SHE driven by the space-time white noise ($H=1/2$) was obtained  in \cite{ANZ98}. There 
the error estimate is O($  k^\frac14+h k^{-\frac14})$. Obviously this error estimate coincides with ours, but only after an additional condition on $k$ and $h$ is enforced. For the  Galerkin finite element approximation $ \tilde u _h$ for the  SHE we have the following error estimate (see Theorem \ref{fem-ord})
\begin{align}\label{eq:error-SHE}
\sup_{t\in I}\bigg(\ee\bigg[\| \tilde u (t)- \tilde u _h(t)\|_{\mathbb L^2}^p\bigg]\bigg)^\frac1p
\le C (h^H+ h^{\frac{3}{2}-\epsilon}k^{-\frac{1}{2}}), 
\end{align}
and for the spectral Galerkin approximation $ \tilde u _N$ for the SWE we have  (see Theorem \ref{spe-ord0})
\begin{align}\label{eq:error-SWE}
\sup_{t\in I}\bigg(\ee\bigg[\| \tilde u (t)- \tilde u _N(t)\|_{\mathbb L^2}^p\bigg]\bigg)^\frac1p
\le C N^{-1} h^{H-1}
\end{align}
where $N$ is the number of terms in the spectral approximation.  We notice that our error estimates  remove a negative infinitesimal component which appears in many error estimates for numerical solutions of  SPDES (see \cite{ANZ98, CY07, DZ02, Yan05}).

The paper is organized as follows. 
In Section \ref{sec2}, we provide some preliminaries about stochastic integrals with respect to white-fractional noise with $H\le 1/2$, followed by establishing the  existence and uniqueness as well as the Sobolev regularity of the mild solution.  In Section \ref{sec3}, we study the regularization of the noise with Wong-Zakai approximation and derive the error estimates \eqref{eq:error-W-Z-SHE} and \eqref{eq:error-W-Z-SWE}.  Finally in Section \ref{sec4}, we apply the Galerkin approximation  to spatially discretize  the regularized SEEs and derive the error estimates  \eqref{eq:error-SHE} and  \eqref{eq:error-SWE}.

\section{Well-posedess and regularity of the SEEs}
\label{sec2}

 In this section, we first introduce the stochastic integral with respect to the white-factional noise $\xi$ with deterministic functions as integrands.
Then we use such  integrals to define the mild solution and establish the well-posedness and Sobolev regularity of Eq. \eqref{see}.

\subsection{Stochastic integrals with respect to the white-fractional noise}  
We follow the approach of \cite{CHL15a} to define the stochastic integrals with respect to the white-fractional noise $\xi$. 
To this end, we introduce a set $\EE$ of all  step functions in $I\times \overline \OO$  in the form of 
\begin{align*}
f(t,x)=\sum_{i=0}^{M-1}\sum_{j=0}^{N-1} f_{ij}\chi_{(a_i, a_{i+1}]\times (b_j, b_{j+1}]}(t,x),\quad (t,x)\in I\times \overline \OO,
\end{align*}
where $0=a_0<a_1<\cdots<a_M=T$ and $0=b_0<b_1<\cdots<b_N=1$ are partitions of $I$ and $\OO$, respectively, and $f_{ij}\in\rr$, $i=0,1,\cdots,M-1$, $i=0,1,\cdots,N-1$, $M,N\in \nn_+$. 
For $f\in \EE$, we define its integral with respect to $W$ by the Riemann sum as
\begin{align*}
&\int_I \int_{\OO} f(s,y) \xi(ds,dy) \\
&:=\sum_{i=0}^{M-1}\sum_{j=0}^{N-1} f_{ij}
\bigg(W(a_{i+1})-W(a_i)\bigg) \bigg(W(b_{j+1})-W(b_j)\bigg),
\end{align*}
and for $f,g\in \EE$, we define their scalar product as
\begin{align*}
\Psi(f,g):=\ee\bigg[\bigg(\int_I\int_{\OO} f(s,y) \xi(ds,dy)\bigg) \bigg(\int_I\int_{\OO}g(s,y) \xi(ds,dy)\bigg)\bigg].
\end{align*}

Next we extend $\EE$ through completion to a Hilbert space, denoted by $\mathcal H$,  and define the  stochastic integral for 
any function $f\in \mathcal H$ accordingly. 
By \cite[Theorem 2.9]{BJQ15},
we have a characterization for $\mathcal H$: 
\begin{align*}
\mathcal H
&=\bigg\{f\ \text{is Lebesgue measurable}:\ \int_I\int_{\OO}\int_{\OO} 
\frac{\left|f(s,x)-f(s,y)\right|^2}{|x-y|^{2-2H}}dxdyds    \nonumber   \\
&\qquad\qquad+\int_I\int_{\OO} f^2(s,x) \bigg( x^{2H-1}+(1-x)^{2H-1} \bigg)dxds<\infty \bigg\}
\end{align*}
and the following It\^{o} isometry (cf. \cite{BJQ15, CHL15a}).
\\

\begin{tm}\label{ito}
For all $f, g \in \mathcal H$,
\begin{align}\label{ito0}
&\ee\bigg[\bigg(\int_I\int_{\OO} f(s,y) \xi(ds,dy)\bigg) 
\bigg(\int_I\int_{\OO}g(s,y) \xi(ds,dy)\bigg)\bigg]    \\
&=\frac{H(1-2H)}{2} \int_I\int_{\OO}\int_{\OO} 
\frac{\big(f(s,x)-f(s,y)\big) \big(g(s,x)-g(s,y)\big)}{|x-y|^{2-2H}}dxdyds    \nonumber   \\
&\quad+H \int_I\int_{\OO} f(s,x)g(s,x)
\bigg( x^{2H-1}+(1-x)^{2H-1} \bigg) dxds.\nonumber 
\end{align}
\end{tm}

We remark that in the case of space-time white noise, i.e., $H=1/2$, \eqref{ito0} becomes the It\^o isometry for space-time white noise:
\begin{align*}
&\ee\bigg[\bigg(\int_I\int_{\OO} f(s,y) \xi(ds,dy)\bigg) 
\bigg(\int_I\int_{\OO}g(s,y) \xi(ds,dy)\bigg)\bigg] \\
&\qquad \qquad \qquad =\int_I\int_{\OO} f(s,y)g(s,y)dyds.
\end{align*}

We will frequently use the following inequalities about the Lebesgue integrals of the singular kernels associated with the fractional Brownian motion (see \cite{CHL15a}, Lemma 2.2 and Lemma 2.3):
\begin{align}
\sum_{j=0}^{n-1}\int_{\OO_j}\int_{\OO_j}|y-z|^{2H-1}dydz
&\le C h^{2H}, \label{same}  \\
\sum_{j\neq l}\int_{\OO_j}\int_{\OO_l}|y-z|^{2H-2}dydz
&\le C h^{2H-1}. \label{dif}
\end{align}

\subsection{Well-posedness and Sobolev regularity}

In this subsection, we study the existence and unique of the mild solution and its regularity for Eq. \eqref{see}. 
The mild solution is defined by the Green's function for the given partial differential operator $L$ which we define as follow. 
Let $\{(\lambda_\alpha,\  \varphi_\alpha)\}_{\alpha\in \nn_+}$, be the eigensystem of the negative Laplacian $-\Delta$. 
Under Dirichlet condition \eqref{dbc}, it is given by
\begin{align*}
\lambda_\alpha:=(\alpha\pi)^2,\quad 
\varphi_\alpha(x):=\sqrt{2}\sin(\sqrt{\lambda_\alpha} x),
\quad  x\in \overline{\OO}, \quad \alpha\in \nn_+. \
\end{align*}
In this paper, all discussions are concerned with the Dirichlet condition \eqref{dbc}. 
However, the main results are also valid for the Neumann condition \eqref{nbc} since the main estimates about $\lambda_\alpha$ in Lemma \ref{lm-re} also hold for $\psi_\alpha(\cdot):=\sqrt{2}\cos(\sqrt{\lambda_\alpha}\cdot)$, $\alpha\in \nn_+$, which are eigenfunctions corresponding to the 
eigenvalues  $\lambda_\alpha$ of  $-\Delta$ with the Neumann boundary condition.

With the above eigensystem, the Green's function for $L$ can be represented as  (cf. \cite{Duf15})
\begin{align}\label{gre}
G_t(x,y)
&=\sum_{\alpha=1}^\infty \phi_\alpha(t)\varphi_\alpha(x)\varphi_\alpha(y),\quad t\in I,\ x,y \in \overline{\OO},
\end{align}
where $\phi_\alpha(t)=e^{-\lambda_\alpha t}$ for SHE and $\phi_\alpha(t)=\frac{\sin(\sqrt{\lambda_\alpha} t)}{\sqrt{\lambda_\alpha}}$ for SWE, $t\in I$.
For convenence, we set $\phi_\alpha(t-s)=0$, and thus $G_{t-s}(x,y)=0$, for any $0\le t<s\le T$ and $x,y\in \overline{\OO}$.

Denote by $S$ the stochastic convolution
\begin{align}\label{con}
S (t,x):=\int_0^t\int_{\OO} G_{t-s}(x,y)\xi(ds,dy),\quad (t,x)\in I\times \overline{\OO}.
\end{align}
Then the mild solution $u$ of Eq. \eqref{see} is defined as the solution of the following stochastic integral equation (cf. \cite{Dal09}):
\begin{align}\label{mild}
u(t,x)
=\omega(t,x)+\int_0^t\int_{\OO} G_{t-s}(x,y)b(u(s,y))dsdy+S (t,x)
\end{align}
for all $(t,x)\in I\times \overline{\OO}$,
where  $\omega$ is the solution of the deterministic evolution equation $Lu=0$ with the same initial and boundary conditions, i.e., for SHE,
\begin{align*}
\omega(t,x)=\int_{\OO} G_t(x,y) u_0(y) dy,
\quad (t,x)\in I\times \overline{\OO},
\end{align*}
and for SWE,
\begin{align*}
\omega(t,x)=\int_{\OO} G_t(x,y) v_0(y) dy
+\int_{\OO} \frac{\partial G_t(x,y)}{\partial t} u_0(y) dy,
\quad (t,x)\in I\times \overline{\OO}.
\end{align*}

The following lemma will be frequently used in the derivation of the regularity for the mild solution and in the error estimates of the  Wong-Zakai approximation. \\

\begin{lm}\label{lm-re}
\begin{enumerate}
\item  [{\rm (i).}]
For any $y,z\in \rr$, there exists a  constant $C$ such that 
\begin{align}\label{sin-sum}
\begin{split}
\sum_{k=1}^\infty \frac{|\varphi_\alpha(y)-\varphi_\alpha(z)|^2}{\lambda_\alpha}
&\le C|y-z|,\\
\sum_{k=1}^\infty \frac{|\psi_\alpha(y)-\psi_\alpha(z)|^2}{\lambda_\alpha}
&\le C|y-z|.
\end{split}
\end{align}
Moreover, for any $\kappa\in (1/2,3/2)$, there exists a  constant $C=C(\kappa)$ such that 
\begin{align}\label{sin-sum1}
\begin{split}
\sum_{k=1}^\infty \frac{|\varphi_\alpha(y)-\varphi_\alpha(z)|^2}{\lambda_\alpha^\kappa}
&\le C|y-z|^{2\kappa-1},\\
\sum_{k=1}^\infty \frac{|\psi_\alpha(y)-\psi_\alpha(z)|^2}{\lambda_\alpha^\kappa}
&\le C|y-z|^{2\kappa-1}.
\end{split}
\end{align}

\item  [{\rm (ii).}]
For any $H<1/2$ and any $\alpha\in \nn_+$, there exists a  constant $C=C(H)$ such that 
\begin{align}\label{sob-int}
\int_{\OO}\int_{\OO} \frac{|\varphi_\alpha(y)-\varphi_\alpha(z)|^2}
{|y-z|^{2-2H}} dydz
\le C \lambda_\alpha^{\frac{1}{2}-H}.
\end{align}

\item [{\rm (iii).}]
For any $t>0$ and any $y,z\in \rr$, there exists a  constant $C=C(T)$ such that 
\begin{align} \label{phi}
\sum_{\alpha=1}^\infty |\varphi_\alpha(y)-\varphi_\alpha(z)|^2 
\bigg( \int_0^t \phi_\alpha^2(t-s)ds \bigg)
\le C |y-z|.
\end{align}

\item  [{\rm (iv).}]
For any $y,z\in \overline{\OO}$, there exists a  constant $C=C(T)$ such that 
\begin{align}\label{gg}
\int_I \int_{\OO}|G_{t-s}(x,y)-G_{t-s}(x,z)|^2 dxds
\le C |y-z|.
\end{align}
\end{enumerate}
\end{lm}

\begin{proof} 
(i).
For any $x,y\in \rr$,
\begin{align*}
& \sum_{\alpha=1}^\infty \frac{|\varphi_\alpha(y)-\varphi_\alpha(z)|^2}{\lambda_\alpha} \\
&\le \sum_{\alpha=1}^\infty \frac{8\wedge 2\lambda_\alpha |y-z|^2}{\lambda_\alpha}
 \le \int_0^\infty \frac{8}{\pi^2 u^2}\wedge 2|y-z|^2 du  \nonumber  \\
&\le \int_0^{\frac{2}{\pi |y-z|}} 2|y-z|^2du+ \int_{\frac{2}{\pi |y-z|}}^\infty \frac{8}{\pi^2 u^2}du
=\frac{8|y-z|}{\pi}.
\end{align*}
This proves the first inequality of \eqref{sin-sum}. The proof of the second inequality and  \eqref{sin-sum1} are analogous. 

(ii). Let $u=\alpha\pi y$ and $v=\alpha\pi z$. Then 
\begin{align}\label{eq:inequality(ii)}
&\int_{\OO}\int_{\OO} \frac{|\varphi_\alpha(y)-\varphi_\alpha(z)|^2}
{|y-z|^{2-2H}}dydz \nonumber \\
&=\frac{2}{\lambda_\alpha^H}\int_0^{\sqrt{\lambda_\alpha}}\int_0^{\sqrt{\lambda_\alpha}}\frac{(\sin u-\sin v)^2}{|u-v|^{2-2H}}dudv.
\end{align}
Set $K_1=\{(u,v)\in [0,\sqrt{\lambda_\alpha}]^2: |u-v|\le 1\}$ and $K_2=\{(u,v)\in [0,\sqrt{\lambda_\alpha}]^2: |u-v|> 1\}$.  It is easy to see that 
\begin{align*}
\int_{K_1} \frac{|\sin u-\sin v|^2}{|u-v|^{2-2H}} dudv 
\le \int_{K_1} |u-v|^{2H}dudv
\le 2\sqrt{\lambda_\alpha}-1.
\end{align*}
On the other hand, since $(\sin u-\sin v)^2\le 4$ when $(u,v)\in K_2$, we have that 
\begin{align*}
& \int_{K_2} \frac{|\sin u-\sin v|^2}{|u-v|^{2-2H}} dudv \\
&\le 4\int_0^{\sqrt{\lambda_\alpha}}
\bigg[\int_0^{v-1} (v-u)^{2H-2}du+\int_{v+1}^{\sqrt{\lambda_\alpha}} (u-v)^{2H-2}du\bigg] dv  \\
&=\frac{4}{H(1-2H)}\lambda_\alpha^H.
\end{align*}
From the the above two estimates and \eqref{eq:inequality(ii)} we obtain 
\begin{align*}
\int_{\OO}\int_{\OO} 
\frac{|\varphi_\alpha(y)-\varphi_\alpha(z)|^2}{|y-z|^{2-2H}} dydz
\le 4\lambda_\alpha^{\frac{1}{2}-H}+\frac{8}{H(1-2H)}
\le C \lambda_\alpha^{\frac{1}{2}-H}, 
\end{align*}
which is  \eqref{sob-int}. 

(iii).  Direct calculations yield
\begin{align*}
\int_0^t \phi_\alpha^2(t-s)ds
=\begin{cases}
\frac{1-e^{-2\lambda_\alpha t}}{2\lambda_\alpha},&\quad {\rm for} \  \text{SHE}, \\
\frac{\sqrt{\lambda_\alpha} t-\sin(2\sqrt{\lambda_\alpha}t)}{2\lambda_\alpha^{3/2}},&\quad {\rm for} \ \text{SWE}.
\end{cases}
\end{align*}
It follows from  \eqref{sin-sum} that 
\begin{align*}
&\sum_{\alpha=1}^\infty 
\bigg( \int_0^t \phi_\alpha^2(t-s)ds \bigg)
|\varphi_\alpha(y)-\varphi_\alpha(z)|^2   \\
&\le C \sum_{\alpha=1}^\infty \frac{|\varphi_\alpha(y)-\varphi_\alpha(z)|^2}{\lambda_\alpha}
\le C |y-z|, 
\end{align*}
which proves  \eqref{phi}.

(iv). 
By \eqref{gre} and orthogonality of $\phi_\alpha$, we have
\begin{align*}
&\int_I \int_{\OO} |G_{t-s}(x,y)-G_{t-s}(x,z)|^2 dxds \\
&=\sum_{\alpha=1}^\infty 
\bigg(\int_I \phi_\alpha^2(t-s) ds\bigg) |\varphi_\alpha(y)-\varphi_\alpha(z)|^2.
\end{align*}
Then \eqref{gg} follows immediately from \eqref{phi}. 
\end{proof}

Denote by $\dot \hh^\beta=\dot\hh^\beta (\OO)$  the usual intepolation space with its norm defined by $\|\cdot\|_{\beta}:=\|(-\Delta)^\frac{\beta}{2}\cdot\|_{\mathbb L^2}$, $\beta\in \rr$. 
In particular, $\dot\hh^0=\mathbb L^2$.
We have the following well-posedness and Sobolev regularity of Eq. \eqref{see}.\\

\begin{tm}\label{wel}
Let $p\ge 2$ and $\beta\in [0,H)$.
Assume that $u_0\in \mathbb L^p(\Omega; \dot \hh^\beta)$ and 
$v_0\in \mathbb L^p(\Omega; \dot \hh^{\beta-1})$. 
Then Eq. \eqref{see} associated with Dirichlet condition \eqref{dbc} or Neumann condition \eqref{nbc} with initial data $u_0$ and/or $v_0$ has a unique mild solution $u$ defined by \eqref{mild}. 
Furthermore, there exists a constant $C=C(p,T,H)$ such that for SHE,
\begin{align}\label{bon-she}
\ee\bigg[\sup_{t \in I}\|u(t)\|_{\beta}^p\bigg]
\le C\bigg(1+\ee\bigg[\|u_0\|_{\beta}^p\bigg]\bigg).
\end{align} 
and for SWE,
\begin{align}\label{bon-swe}
\ee\bigg[\sup_{t \in I}\|u(t)\|_{\beta}^p\bigg]
\le C\bigg(1+\ee\bigg[\|u_0\|_{\beta}^p\bigg]
+\ee\bigg[\|v_0\|_{\beta-1}^p\bigg]\bigg).
\end{align} 
\end{tm}

\begin{proof} 
Substituting \eqref{gre} for $G_{t-s}(x,y)$, using the fact that $S(t)$ is a Gaussian random field,  and applying the It\^o isometry \eqref{ito0}, we obtain 
\begin{align*}
& \ee\bigg[\|S(t)\|_{\mathbb L^2}^p \bigg] \\
&\le C\Bigg[\bigg(\int_{\OO}\int_{\OO} 
\frac{\sum\limits_{\alpha=1}^\infty |\varphi_\alpha(y)-\varphi_\alpha(z)|^2
\Big( \int_0^t \phi_\alpha^2(t-s)ds \Big) } {|y-z|^{2-2H}} dydz \\
&\qquad +\bigg(\int_{\OO} \Big( y^{2H-1}+(1-y)^{2H-1}\Big)dy\bigg) 
\bigg(\sum_{\alpha=1}^\infty \bigg( \int_0^t \phi_\alpha^2(t-s)ds \bigg) \bigg) \Bigg]^\frac p2,
\end{align*}
where $C$ is a constant depending only on $p$ and $H$. 
By the estimation \eqref{phi} and the facts that 
\begin{align*}
&\int_{\OO}\int_{\OO} |y-z|^{2H-1}dydz
=\frac{1}{(H+1)(2H+1)},\\
&\int_{\OO}\Big( y^{2H-1}+(1-y)^{2H-1}\Big)dy
=\frac{1}{H},
\end{align*}
we have
\begin{align*}
\ee\bigg[\|S(t)\|_{\mathbb L^2}^p \bigg] 
&\le C\bigg[1+\sum_{\alpha=1}^\infty 
\bigg( \int_0^t \phi_\alpha^2(t-s)ds \bigg) \bigg]^\frac p2
\end{align*}
For SHE, simple calculations yield that
\begin{align*}
\sum_{\alpha=1}^\infty \bigg( \int_0^t \phi_\alpha^2(t-s)ds \bigg)
=\sum_{\alpha=1}^\infty\frac{1-e^{-2\lambda_\alpha t}}{2\lambda_\alpha}
\le \sqrt{\frac t{2\pi}},
\end{align*}
and for SWE we have
\begin{align*}
\sum_{\alpha=1}^\infty \bigg( \int_0^t \phi_\alpha^2(t-s)ds \bigg)
=\sum_{\alpha=1}^\infty\frac{2\sqrt{\lambda_\alpha}t-\sin(2\sqrt{\lambda_\alpha}t)}{2\lambda_\alpha^{3/2}} 
<\infty.
\end{align*}
This shows that $\ee\left[\|S(t) \|_{\mathbb L^2}^p\right]<\infty$, which in turn ensures the existence of the unique mild solution of Eq. \eqref{see} as well as moments' uniform boundedness through the  standard Picard iteration argument (cf. \cite{Dal09} or \cite[Theorem 3.1]{HL16}).

It remains to show \eqref{bon-she} and \eqref{bon-swe}.
By the orthogonality and uniform boundedness of $\varphi_\alpha$ and  It\^o isometry \eqref{ito0}, there exists a  constant $C=C(p,H)$ such that 
\begin{align*}
&\ee\bigg[\sup_{t \in I}\|S(t)\|_\beta^p\bigg] \\
&\le C\Bigg(\sum_{\alpha=1}^\infty 
\bigg(\ee\bigg[\sup_{t \in I} \bigg|\int_0^t\int_{\OO} \lambda_\alpha^\frac{\beta}{2}\phi_\alpha(t-s)\varphi_\alpha(y)\xi(ds,dy)\bigg|^2\bigg] \bigg) \Bigg)^\frac p2 \\
&\le C\sup_{t \in I}\Bigg[\sum_{\alpha=1}^\infty \lambda_\alpha^\beta 
\bigg(\int_0^t \phi_\alpha^2(t-s) ds \bigg)
\bigg(\int_{\OO}\int_{\OO} \frac{|\varphi_\alpha(y)-\varphi_\alpha(z)|^2}{|y-z|^{2-2H}} dydz \bigg)  \Bigg]^\frac p2 \\
& + C\sup_{t \in I} \Bigg[\bigg(\int_{\OO} \bigg(y^{2H-1}+(1-y)^{2H-1} \bigg) dy\bigg)
\sum_{\alpha=1}^\infty \lambda_\alpha^\beta 
\bigg(\int_0^t \phi_\alpha^2(t-s) ds\bigg) \Bigg]^\frac p2 \\
:&=S_1+S_2.
\end{align*}
For $S_1$, the inequality \eqref{sob-int} yields
\begin{align*}
S_1
\le C \sup_{t \in I} \Bigg[\sum_{\alpha=1}^\infty 
\lambda_\alpha^{\beta-H+\frac12} 
\bigg(\int_0^t \phi_\alpha^2(t-s) ds\bigg)  \Bigg]^\frac p2
\le C \Bigg(\sum_{\alpha=1}^\infty \lambda_\alpha^{\beta-H-\frac{1}{2}}  \Bigg)^\frac p2,
\end{align*}
which converges if and only if $\beta<H$.
The second term $S_2$ can be estimated as
\begin{align*}
S_2
\le C \sup_{t \in I} \Bigg[\sum_{\alpha=1}^\infty \lambda_\alpha^\beta 
\bigg(\int_0^t \phi_\alpha^2(t-s) ds\bigg)  \Bigg]^\frac p2
\le C \Bigg(\sum_{\alpha=1}^\infty \lambda_\alpha^{\beta-1} \Bigg)^\frac p2, 
\end{align*}
which is finite if and only if $\beta<1/2$.
Therefore, for any $\beta\in [0,H)$ we have
\begin{align*}
\ee\bigg[\sup_{t \in I}\|S(t)\|_\beta^p\bigg]<\infty.
\end{align*}
Since $b$ is Lipschitz continuous, the standard arguments imply \eqref{bon-she} and \eqref{bon-swe}.
\end{proof}

We remark that well-posedness results were recently established for $H>1/4$ in \cite{BJQ15} for linear ($b=0$) SEEs  whose diffusion coefficient is given by an affine function $\sigma(u)=a_1u+a_2$ with $a_1,a_2\in \rr$, and in \cite{HHLNT15} for linear SHE where $\sigma(u)$ is differentiable with a Lipschitz derivative and $\sigma(0) = 0$. 
We also note that the authors in \cite{HL16} proved the optimal H\"older regularity for the solution of Eq. \eqref{see} in real line, i.e., $(t,x)\in I\times \rr$.

\section{Wong-Zakai approximations}
\label{sec3}

In this section, we regularize the white-fractional noise $\xi$ through the Wong-Zakai approximation and  establish the rate of convergence of the approximate mild solution of the  SEE with $\xi$ replaced by its Wong-Zakai approximation.

First we recall the Wong-Zakai approximation described in Section \ref{sec1}.  For partitions $\{I_i=(t_i,t_{i+1}],\ t_i=ik,\ i=0,1,\cdots,m-1\}$ and 
$\{\OO_j=(x_j,x_{j+1}],\ x_j=jh,\ j=0,1,\cdots,n-1\}$ of $I$ and $\OO$, with $k=T/m$ and $h=1/n$, the  Wong-Zakai approximation to $\xi(t,x)$ is given by
\begin{align}\label{w'}
\tilde \xi(t,x)=\sum_{i=0}^{m-1}\sum_{j=0}^{n-1}
\bigg[\frac{1}{kh}\int_{I_i}\int_{\OO_j}\xi(ds,dy)\bigg]
\chi_{i,j}(t,x),\ (t,x)\in I\times \overline{\OO}. 
\end{align}
It is easy to see from It\^{o} isometry \eqref{ito} that  $\tilde \xi(t)\in \mathbb H$ a.s., for any $t\in I$, 
moreover, for any $p\ge 2$, there exists a constant $C=C(p,T,H)$ such that 
\begin{align}\label{xi}
\sup_{t\in I}\bigg(\ee\bigg[\|\tilde \xi(t)\|^p_{\mathbb L^2}\bigg]\bigg)^\frac1p
\le C k^{-\frac12} h^{-\frac12}.
\end{align}
%
%
Now we consider  the regularized SEE with $\xi$ replaced by $\tilde \xi$ in Eq. \eqref{see}: 
\begin{align} \label{spde-dis}
L  \tilde u (t,x)=b( \tilde u (t,x))+\tilde \xi(t,x),\quad (t,x)\in I\times \overline{\OO},
\end{align}
with same initial and boundary values. 
Similarly to \eqref{mild}, we define  the mild solution of \eqref{spde-dis} as  $\tilde u$ such that 
\begin{align}\label{mild-dis}
 \tilde u (t,x)
=\omega(t,x)+\int_0^t\int_{\OO} G_{t-s}(x,y)b( \tilde u (s,y))dsdy+\tilde S(t,x)
\end{align}
for all $(t,x)\in I\times \overline{\OO}$,
where $\tilde S$ denotes the approximate convolution:
\begin{align}\label{con-a}
\tilde S(t,x):=\int_0^t\int_{\OO} G_{t-s}(x,y)\tilde \xi(s,y)dsdy,
\quad (t,x)\in I\times \overline{\OO}.
\end{align}
Using   \eqref{w'} we can rewrite $\tilde S(t,x)$  as a stochastic integral:
\begin{align*}
\tilde S(t,x)=\int_I\int_{\OO}
G_{t,s}^{m,n}(x,y) dW(s,y),\quad 
(t,x)\in I\times \overline{\OO},
\end{align*}
where
\begin{align*}
G_{t,s}^{m,n}(x,y)=\sum_{i=0}^{m-1}\sum_{j=0}^{n-1}\frac{\chi_{i,j}(s,y)}{kh} \int_{I_i}\int_{\OO_j}
G_{t-\tau}(x,z)d\tau dz
\end{align*}
for $(t,x),\ (s,y)\in I\times \overline{\OO}$.
As a result,
\begin{align*}
S (t,x)-\tilde S(t,x)
=\int_I\int_{\OO} \Big( G_{t-s}(x,y)-G^{m,n}_{t,s}(x,y) \Big)\xi(ds,dy),
\end{align*}
where $(t,x)\in I\times \overline{\OO}$.

In what follows, we derive an error estimate for the Wong-Zakai approximation of $\xi$ and then establish the convergence rate of the mild solution 
$ \tilde u$ of \eqref{spde-dis} to the mild solution $u$ of \eqref{see} in terms of $k$ and $h$. For this purpose we define,  for $t\in I$, 
\begin{align*}
\Psi_\alpha(t):
&=\sum_{i=0}^{m-1} \int_{I_i} 
\bigg[\int_{I_i} 
\bigg( \phi_\alpha(t-s)- \phi_\alpha(t-\tau) \bigg) d\tau \bigg]^2ds,\\
\Upsilon_\alpha(t):
&=\sum_{i=0}^{m-1} \int_{I_i}\int_{I_i}  \phi_\alpha(t-s)
\bigg( \phi_\alpha(t-s)- \phi_\alpha(t-\tau) \bigg) d\tau ds.
\end{align*}
The following estimations are frequently used in our analysis.\\

\begin{lm}
Let $\phi_{\alpha},  \ \alpha\in \nn_+$ be the basis functions related to SHE. Then there  exists a constant $C=C(T,H)$ such that 
\begin{align}\label{lm-she}
\sup_{t\in I}\sum_{\alpha=1}^\infty\Psi_\alpha(t)
\le C k^\frac{5}{2},
\quad 
\sup_{t\in I}\sum_{\alpha=1}^\infty \Upsilon_\alpha(t)
\le C k^\frac{3}{2}.
\end{align}
\end{lm}

\begin{proof} 
Let  $M$ be the integer such that $t\in [t_{M-1},t_M)$.
Define for $i\in [0,M-1]$,
\begin{align*}
\Psi_i^\alpha(t):=\int_{I_i} \left[\int_{I_i} 
\Big( \phi_\alpha(t-s)- \phi_\alpha(t-\tau) \Big)
d\tau \right]^2ds.
\end{align*}
Then
$\Psi_i^{\alpha}(t)
=\int_{I_i} \Big[\int_{I_i}\int_\tau^s \lambda_\alpha e^{-\lambda_\alpha(t-u)}dud\tau \Big]^2ds$.
When $i\in [0,M-2]$,
\begin{align*}
\Psi_i^{\alpha}(t)
&\le \int_{I_i} \bigg[\int_{I_i}\int_{t_i}^{s\vee \tau} \lambda_\alpha e^{-\lambda_\alpha(t-u)}dud\tau \bigg]^2ds   \\
&\le 2\int_{I_i} \bigg[\int_{I_i}\int_{t_i}^s \lambda_\alpha e^{-\lambda_\alpha(t-u)}dud\tau \bigg]^2ds  \\
&\quad +2\int_{I_i} \bigg[\int_{I_i}\int_{t_i}^\tau \lambda_\alpha e^{-\lambda_\alpha(t-u)}dud\tau \bigg]^2ds     \\
&\le 4k^2\int_{I_i} \bigg[\int_{t_i}^s \lambda_\alpha e^{-\lambda_\alpha(t-u)}du \bigg]^2ds   \\
&\le 2k^2\frac{(1-e^{\lambda_\alpha k})^2 
(1-e^{2\lambda_\alpha k})}{\lambda_\alpha} e^{-2\lambda_\alpha(t-t_i)}.
\end{align*}
Summing up $\Psi_i^{\alpha}(t)$ from $0$ to $M-2$, we obtain 
\begin{align*}
\sum_{i=0}^{M-2} \Psi_i^\alpha(t)
&\le 2k^2\frac{(1-e^{\lambda_\alpha k})^2}{\lambda_\alpha}.
\end{align*}
On the other hand,
\begin{align*}
\Psi_{M-1}^\alpha(t)
&=\int_{t_{M-1}}^t \bigg[\int_{t_{M-1}}^t\int_\tau^s  \lambda_\alpha e^{-\lambda_\alpha(t-u)}dud\tau+\int_t^{t_M}e^{-\lambda_\alpha(t-s)} d\tau\bigg]^2ds    \nonumber \\
&\quad + \int_t^{t_M}\bigg[\int_{t_{M-1}}^t  e^{-\lambda_\alpha(t-\tau)}d\tau\bigg]^2ds 
:=\Psi_{M-1,1}^\alpha(t)+\Psi_{M-1,2}^\alpha(t).
\end{align*}
The first term $\Psi_{M-1,1}^\alpha(t)$ has the estimation:
\begin{align*}
\Psi_{M-1,1}^\alpha(t)
&\le 2\int_{t_{M-1}}^t \bigg[\int_{t_{M-1}}^t\int_\tau^s  \lambda_\alpha e^{-\lambda_\alpha(t-u)}dud\tau \bigg]^2ds \\
&\qquad \qquad +\frac{1-e^{-2\lambda_\alpha(t-t_{M-1})}}{\lambda_\alpha} k^2.
\end{align*}
Similar to the analysis and calculations for $\Psi_i^\alpha(t)$, $i\in [0,M-2]$, the first term on the right hand side above can be controlled by
\begin{align*}
&2\int_{t_{M-1}}^t \bigg[\int_{t_{M-1}}^t\int_\tau^s  \lambda_\alpha e^{-\lambda_\alpha(t-u)}dud\tau \bigg]^2ds \\\
&\le 8k^2\int_{t_{M-1}}^t \bigg[\int_{t_{M-1}}^s \lambda_\alpha e^{-\lambda_\alpha(t-\tau)}d\tau \bigg]^2ds   \\
&\le 4k^2\frac{(1-e^{2\lambda_\alpha k})^3}{\lambda_\alpha}.
\end{align*}
As a result,
\begin{align*}
\Psi_{M-1,1}^\alpha(t)
&\le 4k^2\frac{[1-e^{2\lambda_\alpha k}]^3}{\lambda_\alpha}+ \frac{1-e^{-2\lambda_\alpha(t-t_{M-1})}}{\lambda_\alpha} k^2
\le 5k^2\frac{1-e^{2\lambda_\alpha k}}{\lambda_\alpha}.
\end{align*}
Since $1-e^{-x}\le x$ for any $x\le 0$, 
\begin{align*}
\Psi_{M-1,2}^\alpha(t)
&=\int_t^{t_M}\bigg[\int_{t_{M-1}}^t e^{-\lambda_\alpha(t-\tau)}d\tau \bigg]^2ds
 =\frac{(1-e^{-2\lambda_\alpha(t-t_{M-1})})^2}{\lambda_\alpha} k \\
&\le \frac{1-e^{-2\lambda_\alpha(t-t_{M-1})}}{\lambda_\alpha} k^2
\le \frac{1-e^{2\lambda_\alpha k}}{\lambda_\alpha} k^2.
\end{align*}
Therefore,
$\Psi_{M-1}^\alpha(t)\le 6k^2\frac{1-e^{2\lambda_\alpha k}}{\lambda_\alpha}$.
As a consequence, 
\begin{align*}
\sum_{\alpha=1}^\infty\Psi_\alpha(t)
\le 8k^2\sum_{\alpha=1}^\infty\frac{1-\phi_\alpha(2k)}{\lambda_\alpha}
\le C k^\frac{5}{2}.
\end{align*}

Define for $i\in [0,M-1]$,
\begin{align*}
\Upsilon_i^\alpha(t):=\int_{I_i}\int_{I_i}  \phi_\alpha(t-s)
\Big( \phi_\alpha(t-s)-\phi_\alpha(t-\tau) \Big) d\tau ds.
\end{align*}
When $i\in [0,M-2]$, Similar to the arguments for $\Psi_i^\alpha(t)$, we have
\begin{align*}
\Upsilon_i^\alpha(t)
&=\int_{I_i} e^{-\lambda_\alpha(t-s)}
\bigg[\int_{I_i}\int_s^\tau \lambda_\alpha e^{-\lambda_\alpha(t-u)} dud\tau \bigg] ds\\
&\le \int_{I_i} e^{-\lambda_\alpha(t-s)}
\bigg[\int_{I_i}\int_{t_i}^s \lambda_\alpha e^{-\lambda_\alpha(t-u)} dud\tau\bigg] ds\\
&\quad +\int_{I_i} e^{-\lambda_\alpha(t-s)}
\bigg[\int_{I_i}\int_{t_i}^\tau \lambda_\alpha e^{-\lambda_\alpha(t-u)} dud\tau\bigg] ds\\
&\le 2k \Big(1-e^{-\lambda_\alpha k}\Big)
\bigg(\int_{I_i} e^{-2\lambda_\alpha(t-s)}ds\bigg).
\end{align*}
Summing up $\Upsilon_i^\alpha(t)$ both for $i$ from $0$ to $M-2$ and $\alpha\in \nn_+$, we obtain
\begin{align*}
\sum_{\alpha=1}^\infty \sum_{i=0}^{M-2} \Upsilon_i^\alpha(t)
\le k\sum_{\alpha=1}^\infty\frac{1-e^{-\lambda_\alpha k}}{\lambda_\alpha}
\le C k^\frac{3}{2}.
\end{align*}

On the other hand,
\begin{align*}
\Upsilon_{M-1}^\alpha(t)
&=\int_{t_{M-1}}^t \int_{t_{M-1}}^t \phi_\alpha(t-s)
\Big(\phi_\alpha(t-s)-\phi_\alpha(t-\tau)\Big) d\tau ds  \\
&\quad +\int_{t_{M-1}}^t \int_t^{t_M} \phi_\alpha(t-s) d\tau ds \\
&=(t-t_{M-1})\bigg(\int_{t_{M-1}}^t \phi^2_\alpha(t-s) ds\bigg)
-\bigg(\int_{t_{M-1}}^t \phi_\alpha(t-s) ds\bigg)^2 \\
&\quad +(t_M-t) \bigg(\int_{t_{M-1}}^t \phi_\alpha(t-s) ds\bigg)  \\
&\le (t-t_{M-1}) \bigg(\int_{t_{M-1}}^t \phi^2_\alpha(t-s) ds\bigg) \\
&\quad +(t_M-t) \bigg(\int_{t_{M-1}}^t \phi_\alpha(t-s) ds\bigg). 
\end{align*}
Thus 
\begin{align*}
\sum_{\alpha=1}^\infty \Upsilon_{M-1}^\alpha(t)
&\le k\sum_{\alpha=1}^\infty
\bigg(\frac{1-e^{-2\lambda_\alpha (t-t_{M-1})}}{2\lambda_\alpha}\bigg) \\
&\quad +k\sum_{\alpha=1}^\infty 
\bigg(\frac{1-e^{-\lambda_\alpha (t-t_{M-1})}}{\lambda_\alpha} \bigg)
\le C k^\frac{3}{2}, 
\end{align*}
which completes the proof of \eqref{lm-she}.
\end{proof}

Following the same arguments as in the above lemma,  we have the following estimations for SWE.

\begin{lm}
Let $\phi_{\alpha},  \ \alpha\in \nn_+$ be the basis functions related to SWE. Then there  exists a constant $C=C(T,H)$ such that 
\begin{align}\label{lm-swe}
\sup_{t\in I}\sum_{\alpha=1}^\infty\Psi_\alpha(t)
\le C k^3,
\quad 
\sup_{t\in I}\sum_{\alpha=1}^\infty \Upsilon_\alpha(t)
\le C k^2.
\end{align}
\end{lm}

Applying It\^o isometry \eqref{ito0} and the above two lemmas, we have the following  estimate for the error between  the convolution 
$S$ and $\tilde S$ given by \eqref{con} and \eqref{con-a}, respectively.

\begin{tm}\label{ww} For  $p\geq 2$, there exists a constant $C=C(p,T,H)$ such that for SHE,
\begin{align}\label{ww-she}
\sup_{t\in I}\left(\ee\left[\|S (t)-\tilde S(t)\|_{\mathbb L^2}^p\right] \right)^\frac1p
\le C (h^H+k^\frac14 h^{H-\frac12})
\end{align}
and for SWE,
\begin{align}\label{ww-swe}
\sup_{t\in I}\left(\ee\left[ \|S (t)-\tilde S(t)\|_{\mathbb L^2}^p\right] \right)^\frac1p
\le C(h^H+k^\frac12 h^{H-\frac12}).
\end{align}
\end{tm}

\begin{proof} 
We only prove \eqref{ww-she} for $p=2$, the other cases can be handled by the fact that $S-\tilde S$ is a Gaussian field.
By It\^o isometry \eqref{ito0},
\begin{align*}
\ee\bigg[\|S(t)-\tilde S(t)\|_{\mathbb L^2}^2\bigg] 
=\frac{H(1-2H)}{2}  I_1(t)+H I_2(t),
\end{align*}
where 
\begin{align*}
I_1(t)
&=\int_I \int_{\OO} \int_{\OO}  \int_{\OO} 
\left| G_{t-s}(x,y)-G^{m,n}_{t-s}(x,y)
-G_{t-s}(x,z)+G^{m,n}_{t-s}(x,z) \right|^2 \\
&\qquad \qquad \qquad \qquad \qquad \qquad \qquad \qquad \qquad 
|y-z|^{2H-2}dxdydzds, \\
I_2(t)
&=\int_I \int_{\OO}  \int_{\OO} 
\left| G_{t-s}(x,y)-G^{m,n}_{t-s}(x,y)
\right|^2 \Big( y^{2H-1}+(1-y)^{2H-1}\Big) dxdyds.
\end{align*}
For the first term, we have 
\begin{align*}
I_1(t)&=\sum_{j=0}^{n-1}\int_{\OO_j}\int_{\OO_j} 
\frac{\int_I \int_{\OO}|G_{t-s}(x,y)-G_{t-s}(x,z)|^2 dxds}{|y-z|^{2-2H}} dydz \\
&\quad+\frac{1}{(kh)^2}\sum_{i=0}^{m-1}\sum_{j\neq l}\int_{\OO} \int_{I_i}\int_{\OO_j}\int_{\OO_l} \\
&\qquad \Bigg[\bigg(\int_{I_i}\int_{\OO_j}
\Big(G_{t-s}(x,y)- G_{t-\tau}(x,r)\Big) drd\tau\bigg) |y-z|^{2H-2}  
\\
&\qquad \qquad \times \bigg(\int_{I_i}\int_{\OO_l}\Big(G_{t-s}(x,z)- G_{t-\tau}(x,r)\Big) drd\tau \bigg) \Bigg] dydzdsdx \\
&=:I_{11}(t)+I_{12}(t). 
\end{align*}
By \eqref{gg} and \eqref{same}, we have
\begin{align*}
I_{11}(t)
&\le C \sum_{j=0}^{n-1}\int_{\OO_j}\int_{\OO_j}  |y-z|^{2H-1} dydz
\le C h^{2H}.
\end{align*}
Applying H\"older inequality repetitively and Fubini theorem, we obtain
\begin{align*}
I_{12}(t)
&=\frac{1}{(kh)^2}\sum_{i=0}^{m-1}\sum_{j\neq l}\sum_{\alpha=1}^\infty \int_{I_i}\int_{\OO_j}\int_{\OO_l}  \Bigg[ |y-z|^{2H-2} \nonumber   \\
&\quad \Bigg(\int_{I_i}\int_{\OO_j}
\bigg[\phi_\alpha (t-s)\varphi_\alpha (y)-
 \phi_\alpha (t-\tau_1)\varphi_\alpha (r_1) \bigg]
d\tau_1 dr_1 \Bigg)  \\
&\quad \Bigg( \int_{I_i}\int_{\OO_l}
\bigg[\phi_\alpha (t-s)\varphi_\alpha (z)-
 \phi_\alpha (t-\tau_2)\varphi_\alpha (r_2)\bigg]
d\tau_2 dr_2\Bigg) \Bigg] dydzds\\
&=:I_{121}(t)+I_{122}(t)+I_{123}(t)+I_{124}(t),
\end{align*}
where
\begin{align*}
I_{121}(t)
&=\frac{1}{h^2}\sum_{j\neq l}\sum_{\alpha=1}^\infty 
\bigg( \int_0^t \phi_\alpha^2(t-s)ds \bigg) 
\bigg(\int_{\OO_j}\int_{\OO_l}|y-z|^{2H-2}dydz\bigg)  \\
&\qquad\qquad
\bigg( \int_{\OO_j}[\varphi_\alpha(y)-\varphi_\alpha(r_1)]dr_1\bigg)
\bigg(\int_{\OO_l} [\varphi_\alpha(z)-\varphi_\alpha(r_2)] dr_2\bigg), \\
I_{122}(t)&=\frac{1}{(kh)^2}\sum_{j\neq l} 
\bigg( \int_{\OO_j}\int_{\OO_l} |y-z|^{2H-2} dydz\bigg) \\
&\qquad\qquad \left[\sum_{\alpha=1}^\infty \Psi_\alpha(t) 
\bigg(\int_{\OO_j}\varphi_\alpha(r_1) dr_1\bigg) 
\bigg(\int_{\OO_l}\varphi_\alpha(r_2) dr_2\bigg) \right], \\
I_{123}(t)&=\frac{1}{kh^2}\sum_{\alpha=1}^\infty \Upsilon_\alpha(t)  \bigg( \int_{\OO_l}\varphi_\alpha(r_2) dr_2 \bigg) \\
&\qquad\qquad 
\Bigg(\sum_{j\neq l}\int_{\OO_j}\int_{\OO_l}\int_{\OO_j}  
\frac{\varphi_\alpha(y)-\varphi_\alpha(r_1)}{|y-z|^{2-2H}} dr_1dydz\Bigg), 
\\
I_{124}(t)&=\frac{1}{kh^2}\sum_{\alpha=1}^\infty \Upsilon_\alpha(t)  
\bigg(\int_{\OO_j}\varphi_\alpha(r_1) dr_1 \bigg) \\
&\qquad\qquad 
\Bigg(\sum_{j\neq l}\int_{\OO_j}\int_{\OO_l}\int_{\OO_l} 
\frac{\varphi_\alpha(z)-\varphi_\alpha(r_2)}{|y-z|^{2-2H}} dr_2dydz\Bigg).
\end{align*}
By  Young's inequality,  estimates \eqref{phi} and \eqref{dif},  we have that 
\begin{align*}
& |I_{121}(t)| \\
&\le \frac{C}{h^2} \sum_{j\neq l} 
\int_{\OO_j}\int_{\OO_l} 
\frac{\int_{\OO_j} \sum\limits_{\alpha=1}^\infty 
|\varphi_\alpha(y)-\varphi_\alpha(r_1)|^2 
\Big( \int_0^t \phi_\alpha^2(t-s)ds \Big) dr_1}
{|y-z|^{2-2H}} dydz \\
&\le \frac{C}{h^2} \sum_{j\neq l}\int_{\OO_j}\int_{\OO_l}
\frac{\int_{\OO_j}\int_{\OO_l} \big(|y-r_1|+|z-r_2| \big) dr_1dr_2}
{|y-z|^{2-2H}} dydz  \\
&\le C h \sum_{j\neq l}\int_{\OO_j}\int_{\OO_l}|y-z|^{2H-2}dydz
\le C h^{2H}.
\end{align*}
From  \eqref{lm-she}  and \eqref{dif}, we have, for the second term $I_{122}(t)$,
\begin{align*}
|I_{122}(t)|
&\le C \frac{1}{k^2} 
\bigg(\sum_{\alpha=1}^\infty \Psi_\alpha(t)\bigg) 
\bigg(\sum_{j\neq l}\int_{\OO_j}\int_{\OO_l} |y-z|^{2H-2} dydz\bigg)\\
&\le C k^{-2}h^{2H-1} \bigg(\sum_{\alpha=1}^\infty \Psi_\alpha(t)\bigg)
\le C (h^{2H}+k^\frac{1}{2}h^{2H-1}).
\end{align*}
By \eqref{lm-she}, we have, for the third term $I_{123}(t)$,
\begin{align*}
|I_{123}(t)|
&\le \frac C{kh}
\sum_{j\neq l}\int_{\OO_j}\int_{\OO_l} 
\frac{\sum\limits_{\alpha=1}^\infty \Upsilon_\alpha(t)
\Big(\int_{\OO_j} \Big(\varphi_\alpha(y)-\varphi_\alpha(r_1)\Big) dr_1\Big)}
{|y-z|^{2-2H}} dydz\\
&\le \frac Ck \bigg(\sum_{\alpha=1}^\infty \Upsilon_\alpha(t)\bigg) 
\bigg(\sum_{j\neq l}\int_{\OO_j}\int_{\OO_l} |y-z|^{2H-2}   dydz\bigg) \\
&\le C (h^{2H}+k^\frac{1}{2}h^{2H-1}).
\end{align*}
Analogously, the last term $I_{124}(t)$ satisfies
\begin{align*}
|I_{124}(t)|
&\le C (h^{2H}+k^\frac{1}{2}h^{2H-1}).
\end{align*}
The above four estimates yield
\begin{align*}
\left|I_{12}(t) \right|
\le C (h^{2H}+k^\frac{1}{2}h^{2H-1}).
\end{align*}
Combining the estimations of $I_{11}$ and $I_{12}$, we get
\begin{align} \label{e1}
|I_1(t)|\le C (h^{2H}+k^\frac{1}{2}h^{2H-1}).
\end{align}

Next we estimate $I_2$.  Young's inequality yields
\begin{align*}
I_2(t)
&\le \frac{C}{(kh)^2}\sum_{i=0}^{m-1}\sum_{j=0}^{n-1}
\Bigg[ \int_{\OO}\int_{I_i}\int_{\OO_j} 
\Big( y^{2H-1}+(1-y)^{2H-1} \Big) \\
&\qquad \bigg|\int_{I_i}\int_{\OO_j} 
\Big( G_{t-s}(x,y)-G_{t-s}(x,z) \Big)dzd\tau \bigg|^2 \Bigg] dsdydx   \\
&\quad+\frac{C}{(kh)^2}\sum_{i=0}^{m-1}\sum_{j=0}^{n-1}
\Bigg[ \int_{\OO}\int_{I_i}\int_{\OO_j}  \Big( y^{2H-1}+(1-y)^{2H-1} \Big) \\
&\qquad   \bigg|\int_{I_i}\int_{\OO_j} 
\Big(G_{t-s}(x,z)-G_{t-\tau}(x,z)\Big) dzd\tau \bigg|^2 \Bigg] dsdydx \\
&=:I_{21}(t)+I_{22}(t).
\end{align*}
By \eqref{gre},  orthogonality of $\phi_\alpha$,  and \eqref{phi}, we have
\begin{align*}
I_{21}(t)
&=\frac{C}{h^2}\sum_{j=0}^{n-1}\int_{\OO_j}  
\Bigg[ \sum_{\alpha=1}^\infty \bigg( \int_0^t \phi_\alpha^2(t-s)ds \bigg)  \\
&\qquad \bigg(\int_{\OO_j}|\varphi_\alpha(y)-\varphi_\alpha(z)|dz \bigg)^2  \Big( y^{2H-1}+(1-y)^{2H-1}\Big)\Bigg]  dy      \\
&\le \frac{C}{h}\sum_{j=0}^{n-1}\int_{\OO_j}\int_{\OO_j}   \Bigg(\bigg[\sum_{\alpha=1}^\infty
\bigg( \int_0^t \phi_\alpha^2(t-s)ds \bigg) 
|\varphi_\alpha(y)-\varphi_\alpha(z)|^2\bigg] \\
&\qquad\qquad\qquad\qquad 
\Big( y^{2H-1}+(1-y)^{2H-1}\Big)\Bigg) dydz  \\
&\le  \frac{C}{h}\sum_{j=0}^{n-1}\int_{\OO_j}\int_{\OO_j}   |y-z| \Big( y^{2H-1}+(1-y)^{2H-1} \Big)dydz\le C h.
\end{align*}
Similarly, for $E_{22}(t)$, by \eqref{lm-she}, we have 
\begin{align*}
I_{22}(t)
&=\frac{C}{(kh)^2}\sum_{j=0}^{n-1} \sum_{\alpha=1}^\infty
\Psi_\alpha(t) 
\bigg[\int_{\OO_j} \bigg|\int_{\OO_j}\varphi_\alpha(z)dz\bigg|^2  \\
&\qquad \qquad \qquad \qquad \Big( y^{2H-1}+(1-y)^{2H-1} \Big)dy\bigg]   \\
&\le \frac{C}{k^2} \Bigg[\int_{\OO} \Big( y^{2H-1}+(1-y)^{2H-1} \Big) dy\Bigg]
\Bigg[\sum_{\alpha=1}^\infty \Psi_\alpha(t) \Bigg]
\le C k^\frac12.
\end{align*}
It follows  from the above two estimates that
\begin{align}\label{e2}
I_2(t)
\le C (h+k^\frac{1}{2}).
\end{align}
Combining \eqref{e1} and \eqref{e2}, we obtain  \eqref{ww-she}.
\end{proof}

Now we are ready to estimate the error between the exact solution $u$ of Eq. \eqref{see} and the approximate solution $ \tilde u $ of Eq. \eqref{spde-dis}.

\begin{tm}\label{uu}
Let $p\ge 2$.
Assume that $u_0\in \mathbb L^p(\Omega; \mathbb L^2)$ and 
$v_0\in \mathbb L^p(\Omega; \dot{\hh}^{-1})$.
Let $u$ and $\tilde u$ be the mild solutions of Eq. \eqref{see} and Eq. \eqref{spde-dis}, respectively. 
Then there exists a constant $C=C(p,T,H,u_0,v_0)$ such that for SHE,
\begin{align}\label{uu-she}
\sup_{t\in I}\bigg(\ee\bigg[ \|u (t)-\tilde u(t)\|_{\mathbb L^2}^p\bigg] \bigg)^\frac1p
\le C (h^H+k^\frac14 h^{H-\frac12})
\end{align}
and for SWE,
\begin{align}\label{uu-swe}
\sup_{t\in I}\bigg(\ee\bigg[ \|u (t)-\tilde u(t)\|_{\mathbb L^2}^p\bigg] \bigg)^\frac1p
\le C(h^H+k^\frac12 h^{H-\frac12}).
\end{align}
\end{tm}

\begin{proof} By Theorem \ref{ww} and the Lipschitz continuity of the drift coefficient function $b$, it suffices to prove that
\begin{align}\label{uw}
\ee\bigg[\|u(t)-\tilde u(t)\|_{\mathbb L^2}^p\bigg]
\le C \ee\bigg[\|S(t)-\tilde S(t)\|_{\mathbb L^2}^p\bigg].
\end{align}
Subtracting \eqref{mild-dis} from \eqref{mild}, we get
\begin{align*}
u(t,x)- \tilde u (t,x)
&=\int_0^t\int_{\OO} G_{t-s}(x,y)
\Big(b(u(s,y))-b( \tilde u (s,y))\Big) dsdy \\
&\quad +S (t,x)-\tilde S(t,x).
\end{align*}
Taking $\mathbb L^2$-norm and then the expectation in the above equation, we have
\begin{align*}
&\ee\bigg[\|u(t)-\tilde u(t)\|_{\mathbb L^2}^p\bigg]   \\
&\le C \ee\Bigg[\bigg\|\int_0^t\int_{\OO} G_{t-s}(\cdot,y)
\Big( b(u(s,y))-b(\tilde u (s,y)) \Big) dsdy\bigg\|_{\mathbb L^2}^p\Bigg] \\
&\quad +C\ee\bigg[\|S(t)-\tilde S(t)\|_{\mathbb L^2}^p\bigg] \\
&\le C \ee \Bigg[\bigg(\int_0^t \sum_{\alpha=1}^\infty \phi^2_\alpha(t-s) (\varphi_\alpha, u(s)- \tilde u (s) )^2ds \bigg)^\frac p2 \Bigg] \\
&\quad +C\ee\bigg[\|S(t)-\tilde S(t)\|_{\mathbb L^2}^p\bigg] \\
&\le C \int_0^t \bigg(\sup_{\alpha\in \nn_+}\phi_\alpha^p(t-s)\bigg)
\ee\bigg[\|u(s)-\tilde u(s)\|_{\mathbb L^2}^p\bigg] ds \\
&\quad +C\ee\bigg[\|S(t)-\tilde S(t)\|_{\mathbb L^2}^p\bigg].
\end{align*}
Since $\sup\limits_{t\in I}\sup\limits_{\alpha\in \nn_+}|\phi_\alpha(t)|\le 1$, we obtain \eqref{uw} by Gronwall's inequality.
\end{proof}

\begin{rk}\label{rk-wz}
(i) If $\xi$ reduces to the space-time white noise, i.e., $H=1/2$, then for SHE we have
\begin{align*}
\sup_{t\in I}\bigg(\ee\bigg[ \|u (t)-\tilde u(t)\|_{\mathbb L^2}^p\bigg] \bigg)^\frac1p
\le C (h^\frac12+k^\frac14),
\end{align*}
and for SWE we have
\begin{align*}
\sup_{t\in I}\bigg(\ee\bigg[ \|u (t)-\tilde u(t)\|_{\mathbb L^2}^p\bigg] \bigg)^\frac1p
\le C(h^\frac12+k^\frac12),
\end{align*}
which shows that $\tilde u$ converges to $u$ as $h,k\rightarrow 0^+$ without any restriction on $h$ and $k$. 
This estimation for SHE improves related result in \cite[Theorem 2.3]{ANZ98} where the authors proved the convergence of 
-Zakai approximation under the assumption that $h/k^\frac{1}{4}\rightarrow 0$:
\begin{align*}
\sup_{t\in I}\bigg(\ee\bigg[ \|u (t)-\tilde u(t)\|_{\mathbb L^2}^2\bigg] \bigg)^\frac12
&\le C(k^\frac14+h k^{-\frac14}).
\end{align*}
Similar result for SWE had also been established in \cite[Theorem 2]{CY07}.

(ii) When $H<1/2$, the Wong-Zakai approximation converges only if 
$k/h^{2-4H}\rightarrow 0$ for SHE and
$k/h^{1-2H}\rightarrow 0$ for SWE. 
Moreover, the convergence rate is optimal if we set
$k=h^2$ for SHE and $k=h$ for SWE:
\begin{align*}
\sup_{t\in I}\bigg(\ee\bigg[ \|u (t)-\tilde u(t)\|_{\mathbb L^2}^p\bigg] \bigg)^\frac1p
\le C h^H.
\end{align*}

(iii) In many literatures (see \cite{KLS10, Yan05} for examples), the order of convergence of various numerical discretizations for SEEs  includes  a negative infinitesimal component $\epsilon$. Our results remove this factor. 
\end{rk}

\section{Galerkin approximations for regularized equations}
\label{sec4}

In this section, we apply the Galerkin  methods to spatially discretize  the regularized equation \eqref{spde-dis} and conduct the error estimates. Specifically, we apply the Galerkin finite element method  to discretize the regularized SHE and the 
spectral Galerkin method  to discretize the regularized SWE.

\subsection{Galerkin finite element method for regularized SHE}

Let $V_h\subset \dot{\mathbb H}^1$ be a family of linear finite element spaces, i.e., $V_h$ consists of continuous piecewise affine polynomials with respect to the same  partition $\{\OO_j\}_{i=0}^{n-1}$ of $\OO$ as in Section \ref{sec3}. 
To introduce the finite element formulation for the regularized SHE \eqref{spde-dis}, we use $P_h:L^2\rightarrow V_h$ to denote the orthogonal projection operator defined by $(P_h u,v)=(u,v)$ for any $u\in \mathbb L^2$ and $v\in V_h$,  and $R_h:L^2\rightarrow V_h$ to denote the Ritz projection operator defined by
$(\nabla R_hu,  \nabla v)=(\nabla u,\nabla v)$ for any $u\in \dot{\mathbb H}^1$ and $v\in V_h$,  where $(\cdot,\cdot)$ denotes the inner product in $\mathbb L^2$. 
Then the semidiscrete Galerkin finite element approximation for Eq. \eqref{spde-dis} is to find an $\tilde u _h\in V_h$ such that $ \tilde u _h(0)=P_h u_0$ and for any $t>0$, 
\begin{align}\label{fem}
d \tilde u _h(t)
=\Delta_h  \tilde u _h(t)dt+P_h\Big(b( \tilde u _h(t))+\tilde \xi(t)\Big) dt,
\end{align}
where $\Delta_h$ is the discrete analogue of Dirichlet Laplacian defined by
$(-\Delta_h u,v)=(\nabla u, \nabla v)$ for any $u,v\in V_h$. 

Define $E(t)=e^{t \Delta}$ and $E_h(t)=e^{t \Delta_h}$ for $t\geq 0$. 
Then Eq. \eqref{spde-dis} and Eq. \eqref{fem} admit  unique mild solutions
\begin{align}\label{fem-mild0}
u (t)=E(t)u_0+\int_0^t E(t-s) \Big(b( \tilde u (s))+\tilde \xi(s)\Big) ds,
\quad t\in I
\end{align}
and  respectively,
\begin{align}\label{fem-mild}
 \tilde u _h(t)
=E_h(t)P_h u_0+\int_0^t E_h(t-s)P_h \Big(b( \tilde u _h(s))+\tilde \xi(s)\Big)ds,
\quad t\in I.
\end{align}

We note that the solution of homogenous equation $\partial_t u=\Delta u$ under the  Dirichlet  boundary condition \eqref{dbc} with initial data $v$ is smooth for positive time due to the fact that the solution operator $E(t)$ of the initial value problem is an analytic semigroup satisfying (cf. \cite[Lemma 3.2]{Tho06})
\begin{align}\label{sem-smo}
\|E(t)v\|_\beta\le C t^{-\frac{\beta-\alpha}{2}} \|v\|_\alpha, 
\quad t>0,\ 0\le \alpha\le \beta.
\end{align}
Define $F_h(t):=E(t)-E_h(t)P_h$.
The following error estimate will play an important role in our error analysis (cf. \cite{Tho06}, Theorem 3.5):
\begin{align}\label{smo}
\|F_h(t) v\|_{\mathbb L^2}
\le C h^\beta t^{-\frac{\beta-\alpha}{2}} \|v\|_\alpha,
\quad t>0,\ 0\le \alpha\le \beta\le 2.
\end{align}

\begin{tm}\label{fem-ord}
Let $p\ge 2$.
Assume that $u_0\in \mathbb L^p(\Omega; \dot \hh^H)$.
Then for any $\epsilon>0$, there exists a constant $C=C(p,T,H,u_0)$ such that 
\begin{align}\label{fem-ord1}
\sup_{t\in I}\bigg(\ee\bigg[\| \tilde u (t)- \tilde u _h(t)\|_{\mathbb L^2}^p\bigg]\bigg)^\frac1p
\le C \Big(h^H+\epsilon^{-1}h^{\frac32-\epsilon} k^{-\frac12}\Big).
\end{align}
\end{tm}

\begin{proof} Denote by $e_h(t):= \tilde u (t)- \tilde u _h(t)$. 
Subtracting \eqref{fem-mild0} from \eqref{fem-mild}, we have 
\begin{align}\label{fem1}
e_h(t)
&=F_h(t)u_0+\int_0^t F_h(t-s)\tilde \xi(s)ds+\int_0^t F_h(t-s)b( \tilde u (s))ds  \nonumber \\
&\quad +\int_0^t E_h(t-s)P_h \Big( b( \tilde u (s))-b( \tilde u _h(s))\Big)ds.
\end{align}
The finite element estimate \eqref{smo} with $\alpha=\beta=H$ yields
\begin{align*}
\|F_h(t)u_0\|_{\mathbb L^p(\Omega; \mathbb L^2)}
\le C h^H \|u_0\|_{\mathbb L^p(\Omega; \dot{\hh}^H)}
\le C h^H.
\end{align*}
By Minkowski's inequality,  \eqref{smo} with $\alpha=0$ and $\beta=2-\epsilon$, and  \eqref{xi}, we have 
\begin{align*}
&\left\|\int_0^t F_h(t-s)\tilde \xi(s)ds\right\|_{\mathbb L^p(\Omega; \mathbb L^2)}  \\
&\le C \int_0^t \left\|F_h(t-s)\tilde \xi(s)\right\|_{\mathbb L^p(\Omega; \mathbb L^2)} ds\nonumber \\
&\le C \sup_{t\in I}\|\tilde \xi(t)\|_{\mathbb L^p(\Omega; \mathbb L^2)}
\sup_{t\in I} \left(\int_0^t h^{2-\epsilon} (t-s)^{-(1-\frac{\epsilon}{2})}ds\right) \\
&\le C \epsilon^{-1}h^{\frac32-\epsilon} k^{-\frac12}.
\end{align*}
Similarly, by the Lipschitz continuity of $b$ and \eqref{bon-she} we have
\begin{align*}
\left\|\int_0^t F_h(t-s)b( \tilde u (s))ds\right\|_{\mathbb L^p(\Omega; \mathbb L^2)}
\le C \epsilon^{-1} h^{2-\epsilon}.
\end{align*}
To estimate the  last term on the right hand side of \eqref{fem1}, we note that, from  the smoothness property \eqref{sem-smo} and the finite element estimate \eqref{smo}, 
$\|E_h(t)P_hv\| \le C \|v\|$.
As a consequence,
\begin{align*}
&\left\|\int_0^t E_h(t-s)P_h \Big(b( \tilde u (s))-b( \tilde u _h(s)) \Big) ds \right\|_{\mathbb L^p(\Omega; \mathbb L^2)} \\
&\le C \int_0^t \|e_h(s)\|_{\mathbb L^p(\Omega; \mathbb L^2)} ds.
\end{align*}
Combining the above estimations, we obtain 
\begin{align*}
\sup_{t\in I}\|e_h(t)\|_{\mathbb L^p(\Omega; \mathbb L^2)}
\le C \Big(h^H+\epsilon^{-1}h^{\frac32-\epsilon} k^{-\frac12} \Big)
+C \int_0^t \|e_h(s)\|_{\mathbb L^p(\Omega; \mathbb L^2)} ds,
\end{align*}
from which we conclude \eqref{fem-ord1} by Gronwall's inequality.
\end{proof}

\begin{rk}
(i) In the derivation of the convergence rate of Galerkin finite element approximation for SHE, we do not need a priori regularity information about 
$ \tilde u $ such as the estimates of $\|\partial_t \tilde u (t)\|_{\mathbb L^p(\Omega; \mathbb L^2)}$ and $\|\tilde u (t)\|_{\mathbb L^p(\Omega; \dot{\hh}^2)}$ needed in \cite{ANZ98, DZ02}. 

(ii) Moreover, when $H<1/2$ and $k=h^2$, we take $\epsilon=1/2-H$ and then 
\begin{align*}
\sup_{t\in I}\bigg(\ee\bigg[\| \tilde u (t)- \tilde u _h(t)\|_{\mathbb L^2}^p\bigg]\bigg)^\frac1p
\le C h^H.
\end{align*}
\end{rk}

\subsection{Spectral Galerkin method for regularized SWE}

In this subsection, we use the spectral Galerkin method to spatially discretize  the regularized SWE and derive its rate of convergence. 

Let $V_N:=\text{span}\{\varphi_\alpha\}_{\alpha=1}^N$ and $P_N:L^2\rightarrow V_N$ the correponding orthogonal projection operator  defined by $(P_Nu,v)=(u,v)$ for any $v\in V_N$. Then the spectral Galerkin method for Eq. \eqref{spde-dis} is to find an $\tilde u _N\in V_N$ such that $ \tilde u _N(0)=P_Nu_0$ and for any $t>0$, $v\in V_N$,
\begin{align}\label{spe}
(\partial_t \tilde u _N(t),v)=(P_Nv_0,v)+\int_0^t ( \tilde u _N(s), \Delta v)ds+\int_0^t (b( \tilde u _N(s))+\tilde \xi(s), v)ds.
\end{align}

It is well-known that $\|P_N\|=1$, $\Delta P_N u=P_N \Delta u$ for any $u\in \dot{\hh}^2$ and $\|u-P_N u\|_{\mathbb L^2} \le N^{-1}\|u\|_1$ for any $u\in \dot{\hh}^1$ (cf. \cite[Lemma 4]{CY07}).  To estimate the error between  $ \tilde u (t)$ and $ \tilde u _N(t)$ we first split it into two parts as follows. 
\begin{align}\label{div}
 \tilde u (t)- \tilde u _N(t)
=\Big(\tilde u (t)-P_N  \tilde u (t)\Big)
+\Big(P_N  \tilde u (t)- \tilde u _N(t)\Big).
\end{align}
The first part has the following estimation.

\begin{prop}\label{reg}
Let $p\ge 2$. 
Assume that $u_0\in \mathbb L^p(\Omega; \dot{\hh}^1)$ and $v_0\in \mathbb L^p(\Omega; \mathbb L^2)$.
Then the approximate solution $\tilde u$ of the regularized SWE is in $L^\infty([0,T]; \mathbb L^p(\Omega; \dot\hh ^1))$. Moreover, there exists a constant $C=C(p,T,H,u_0,v_0)$ such that
\begin{align}\label{u2}
\sup_{t\in I}\bigg(\ee\bigg[\| \tilde u (t)-P_N  \tilde u (t)\|_{\mathbb L^2}^p\bigg]\bigg)^\frac1p
\le CN^{-1} h^{H-1}. 
\end{align}
\end{prop}

\begin{proof} 
By the above property of $P_N$, we get
\begin{align*}
\|\tilde u (t)-P_N \tilde u (t)\|_{\mathbb L^2}
\le N^{-1}\|\tilde u (t)\|_1
\end{align*}
Thus to prove \eqref{u2} it suffices to show that
\begin{align}\label{u20}
\sup_{t\in I} \bigg(\ee\bigg[\|\tilde u (t)\|_1^p\bigg] \bigg)^\frac1 p
\le C h^{H-1}. 
\end{align}
The definition of the $\dot \hh ^1$-norm and Young's inequality yield
\begin{align}
& \ee\bigg[\|\tilde u(t)\|_1^p\bigg] \nonumber  \\
&\le C\ee\bigg[\bigg(\sum_{\alpha=1}^\infty \lambda_\alpha  
(\varphi_\alpha, \omega(t,\cdot))^2 \bigg)^\frac p2\bigg] \nonumber  \\
&\quad+C\ee \Bigg[\bigg( \sum_{\alpha=1}^\infty\lambda_\alpha
\bigg(\varphi_\alpha, \int_0^t\int_{\OO} G_{t-s}(\cdot,y)b( \tilde u (s,y))dsdy\bigg)^2 \bigg)^\frac p2\Bigg] \nonumber \\
 &\qquad  +C\ee \Bigg[\bigg( \sum_{\alpha=1}^\infty\lambda_\alpha
\bigg(\varphi_\alpha, \int_0^t\int_{\OO} G_{t-s}(\cdot,y)\tilde \xi(ds,dy)\bigg)^2 \bigg)^\frac p2 \Bigg] \nonumber  \\
& =: F_1(t)+F_2(t)+F_3(t).\label{u21}
\end{align}

For SWE, it is clear that 
$\sup\limits_{t\in I}|\phi_\alpha(t)|\le \lambda_\alpha^{-\frac12}$ and 
$\sup\limits_{t\in I}|\phi_\alpha'(t)| \le 1$ for any $\alpha\in \nn_+$.
Then
\begin{align}
& F_1(t)   \\
&=C\ee \Bigg[\bigg(\sum_{\alpha=1}^\infty\lambda_\alpha 
\bigg(\varphi_\alpha, \int_{\OO} G_t(\cdot,y)v_0(y)dy+\int_{\OO} \frac{\partial}{\partial t}G_t(\cdot,y) u_0(y)dy\bigg)^2 \bigg)^\frac p2\Bigg] \nonumber  \\
&\le C\ee\Bigg[\bigg(\sum_{\alpha=1}^\infty \lambda_\alpha \phi^2_\alpha(t) (\varphi_\alpha,v_0)^2\bigg)^\frac p2 \Bigg]
+C\ee\Bigg[\bigg(\sum_{\alpha=1}^\infty \lambda_\alpha |\phi'_\alpha(t)|^2 
(\varphi_\alpha, u_0)^2\bigg)^\frac p2 \Bigg] \nonumber  \\
&\le C\ee\Bigg[\bigg(\sum_{\alpha=1}^\infty (\varphi_\alpha,v_0)^2\bigg)^\frac p2 \Bigg]
+C\ee\Bigg[\bigg(\sum_{\alpha=1}^\infty \lambda_\alpha 
(\varphi_\alpha, u_0)^2\bigg)^\frac p2 \Bigg] \nonumber  \\
&\le C \bigg(\ee\bigg[\|u_0\|_1^p\bigg]
+\ee\bigg[\|v_0\|_{\mathbb L^2}^p\bigg]\bigg).\nonumber
\end{align}
Following a similar argument and combining the Lipschitz continuity of $b$ and the estimation \eqref{bon-swe}, we get
\begin{align}
F_2(t)
&=C\ee\Bigg[\bigg(\sum_{\alpha=1}^\infty\lambda_\alpha \int_0^t \phi^2_\alpha(t-s)
(\varphi_\alpha, b(\tilde u (s)))^2  ds\bigg)^\frac p2\Bigg] \nonumber  \\
&\le \ee\bigg[\bigg( \int_0^t \|b(\tilde u (s))\|_{\mathbb L^2}^2  ds\bigg)^\frac p2\bigg] \nonumber \\
&\le C \bigg(1+\ee\bigg[\|u_0\|_{\mathbb L^2}^p\bigg]
+\ee\bigg[\|v_0\|_{-1}^p\bigg]\bigg) . \label{f2}
\end{align}
For the last term $F_3(t)$, since $\int_0^t\int_{\OO} G_{t-s}(x,y)\tilde \xi(ds,dy)$ is Gaussian, we have
\begin{align*}
F_3(t)
&\le C\Bigg(\frac{1}{(kh)^2}\sum_{\alpha=1}^\infty\lambda_\alpha
\ee\bigg[\bigg|\sum_{i=0}^{m-1}\sum_{j=0}^{n-1} \int_{I_i}\int_{\OO_j}    \widehat{g}^{i,j}_{t,x}(\tau,z)\xi(ds,dy)\bigg|^2\bigg]\Bigg)^\frac p2, 
\end{align*}
where $\widehat{g}^{i,j}_{t,x}(s,y)
=\left(\int_{I_i} \phi_\alpha(t-\tau)d\tau\right) \left(\int_{\OO_j}\varphi_\alpha(z)dz\right)$.
Similar to the estimate bwtween $S$ and $\tilde S$ in Theorem \ref{ww}, we have
\begin{align*}
& F_3(t) \\
&\le\frac C{(kh)^p}\Bigg(\sum_{\alpha=1}^\infty\lambda_\alpha\sum_{i=0}^{m-1}\sum_{j=0}^{n-1}\int_{I_i}\int_{\OO_j}\int_{\OO_j}
\frac{ |\widehat{g}_{t,x}^{i,j}(s,y)-\widehat{g}_{t,x}^{i,j}(s,z)|^2} 
{|y-z|^{2-2H}} dydzds\Bigg)^\frac p2    \nonumber   \\
&\ +\frac C{(kh)^p} \Bigg(\sum_{\alpha=1}^\infty\lambda_\alpha\sum_{i=0}^{m-1}\sum_{j=0}^{n-1}\int_{I_i}\int_{\OO_j} |\widehat{g}_{t,x}^{i,j}(s,y)|^2 
\Big( y^{2H-1}+(1-y)^{2H-1}\Big)dsdy \Bigg)^\frac p2   \nonumber   \\
&\quad+\frac C{(kh)^p} \Bigg(\sum_{\alpha=1}^\infty\lambda_\alpha\sum_{i=0}^{m-1}\sum_{j\neq l}\int_{I_i}\int_{\OO_j}\int_{\OO_l}  
\frac{|\widehat{g}_{t,x}^{i,j}(s,y)|\cdot |\widehat{g}_{t,x}^{i,k}(s,z)|}
{|y-z|^{2-2H}}  dydzds \Bigg)^\frac p2   \nonumber\\
&=:F_{31}(t)+F_{32}(t)+F_{33}(t).
\end{align*}

The first term $F_{31}(t)$ vanishes:
\begin{align*}
F_{31}(t)=0.
\end{align*}

It follows from the inequalities \eqref{phi} and \eqref{sin-sum} that
\begin{align*}
F_{32}(t)
&\le \frac C{h^p}\Bigg[\sum_{j=0}^{n-1}\sum_{\alpha=1}^\infty\lambda_\alpha 
\bigg( \int_0^t \phi_\alpha^2(t-s)ds \bigg)
\bigg(\int_{\OO_j}\varphi_\alpha(z)dz \bigg)^2   \\
&\qquad\qquad \bigg(\int_{\OO_j} \Big( y^{2H-1}+(1-y)^{2H-1} \Big)dy\bigg) \Bigg]^\frac p2   \\
&\le \frac C{h^p} \Bigg[\sum_{j=0}^{n-1} 
\bigg(\sum_{\alpha=1}^\infty \frac{[\psi_\alpha(x_{j+1})-\psi_\alpha(x_j)]^2}{\lambda_\alpha} \bigg)   \\
&\qquad \qquad \bigg(\int_{\OO_j} \Big( y^{2H-1}+(1-y)^{2H-1} \Big)dy\bigg)\Bigg]^\frac p2 
\le C h^{-\frac p2}.
\end{align*}

By \eqref{sin-sum} in Lemma \ref{lm-re} and the estimate \eqref{dif}, we have
\begin{align*}
F_{33}(t)
&\le\frac C{h^p} \Bigg(\sum_{\alpha=1}^\infty\lambda_\alpha  
\bigg( \int_0^t \phi_\alpha^2(t-s)ds \bigg)     \\
&\qquad \qquad\qquad \Bigg[\sum_{j\neq l}\bigg(\bigg(\int_{\OO_j}\varphi_\alpha(r)dr\bigg)^2+\bigg(\int_{\OO_l}\varphi_\alpha(r)dr\bigg)^2\bigg)  \\
&\qquad \qquad\qquad \qquad \qquad 
\bigg(\int_{\OO_j}\int_{\OO_l} |y-z|^{2H-2}  dydz\bigg) \Bigg] \Bigg)^\frac p2   \\
&\le \frac C{h^p}\Bigg[ \sum_{j\neq l}
 \bigg(\int_{\OO_j}\int_{\OO_l} |y-z|^{2H-2}  dydz\bigg) \\
 &\qquad  \bigg(\sum_{\alpha=1}^\infty 
\frac{|\psi_\alpha(x_{j+1})-\psi_\alpha(x_j)|^2}{\lambda_\alpha}
+\sum_{\alpha=1}^\infty 
\frac{|\psi_\alpha(x_{l+1})-\psi_\alpha(x_l)|^2}{\lambda_\alpha}\bigg) 
\Bigg]^\frac p2 \\
&\le C h^{(H-1)p}.
\end{align*}
Thus we obtain
\begin{align}\label{f33}
F_3(t)\le C h^{(H-1)p}.
\end{align}
Combining the estimations \eqref{u21}-\eqref{f33}, we conclude \eqref{u20}.
\end{proof}

\begin{tm}\label{spe-ord0}
Let $p\ge 2$.
Assume that $u_0\in \mathbb L^p(\Omega; \dot{\hh}^1)$ and $v_0\in \mathbb L^p(\Omega; \mathbb L^2)$. 
Then there exists a constant $C=C(p,T,H,u_0,v_0)$ such that
\begin{align}\label{spe-ord00}
\sup_{t\in I}\bigg(\ee\bigg[\| \tilde u (t)- \tilde u _N(t)\|_{\mathbb L^2}^p\bigg]\bigg)^\frac1p
\le C N^{-1} h^{H-1}.
\end{align}
\end{tm}

\begin{proof} The weak formulation of \eqref{spde-dis} reads for any $v\in V_n$,
\begin{align*}
(\partial_t  \tilde u (t),v)=(v_0,v)+\int_0^t ( \tilde u (s), \Delta v)ds+\int_0^t (b( \tilde u (s))+\tilde \xi(s), v)ds.
\end{align*}
Since $\Delta P_N u=P_N \Delta u$ for any $u\in V$, we get for any $v\in V_N$,
\begin{align}\label{wea-pro}
(P_N\partial_t  \tilde u (t),v)
&=(P_Nv_0,v)+\int_0^t (P_N \tilde u (s), \Delta v)ds \nonumber  \\
&\quad +\int_0^t (b( \tilde u (s))+\tilde \xi(s), v)ds.
\end{align}
Set $v(t)=\partial_t (P_N \tilde u (t)-  \tilde u _N(t))$. Then $v(t)$ is an element of $V_N$ for any $t>0$. Subtracting \eqref{wea-pro} from \eqref{spe}, we obtain
\begin{align*}
&\|\partial_t (P_N \tilde u (t)-  \tilde u _N(t))\|_{\mathbb L^2}^2 \\
&=\int_0^t (P_N \tilde u (s)-  \tilde u _N(s), \Delta [\partial_t (P_N \tilde u (s)-  \tilde u _N(s))])ds  \\
&\quad+\int_0^t (b( \tilde u (s))-b( \tilde u _N(s)), \partial_t (P_N \tilde u (s)-  \tilde u _N(s)))ds   \\
&=-\|\nabla (P_N \tilde u (t)-\tilde u _N(t))\|_{\mathbb L^2}^2 \\
&\quad +\int_0^t (b( \tilde u (s))-b( \tilde u _N(s)), \partial_t (P_N \tilde u (s)-  \tilde u _N(s)))ds.
\end{align*}
By Cauchy-Schawarz inequality and the Lipschitz continuity of $b$, we have
\begin{align*}
&\|\partial_t (P_N \tilde u (t)-  \tilde u _N(t))\|_{\mathbb L^2}^2
+\|\nabla (P_N \tilde u (t)-  \tilde u _N(t))\|_{\mathbb L^2}^2   \\
&\le C \int_0^t \|\partial_t (P_N \tilde u (s)-  \tilde u _N(s))\|_{\mathbb L^2}^2 ds
+\int_0^t \| \tilde u (s)- \tilde u _N(s)\|_{\mathbb L^2}^2 ds,
\end{align*}
which in turn, by the classical Gronwall's inequality, yields
\begin{align*}
\|\nabla (P_N \tilde u (t)-  \tilde u _N(t))\|_{\mathbb L^2}^2
\le C \int_0^t \| \tilde u (s)- \tilde u _N(s)\|_{\mathbb L^2}^2 ds.
\end{align*}
By Poincar\'{e}'s inequality we have that 
\begin{align}\label{spe2}
\|P_N \tilde u (t)-  \tilde u _N(t)\|_{\mathbb L^2}^2
& \le C \|\nabla (P_N \tilde u (t)-  \tilde u _N(t))\|_{\mathbb L^2}^2 \nonumber \\
& \le C \int_0^t \| \tilde u (s)- \tilde u _N(s)\|_{\mathbb L^2}^2 ds.
\end{align}
Combining \eqref{div}, \eqref{u2} and \eqref{spe2}, we conclude \eqref{spe-ord00} by Gronwall's inequality.
\end{proof}

\begin{rk}
For the space-time white noise, i.e., $H=1/2$, the above convergence result improves slightly the corresponding result of \cite[Theorem 4]{CY07}, where the authors proved that 
\begin{align*}
\sup_{t\in I} \ee\bigg[\| \tilde u (t)- \tilde u _N(t)\|_{\mathbb L^2}^2\bigg] 
\le C h |\ln h| 
\end{align*}
provided $N=h^{-1}$ and $u_0\in \mathbb L^2(\Omega; \dot{\hh}^{\beta+1})$ and $v_0\in \mathbb L^2(\Omega; \dot{\hh}^\beta)$ for some $\beta>1/2$.
\end{rk}

\bibliographystyle{amsalpha}
\bibliography{bib.bib}
\end{document}